\title{Computing the Fixed Field of $\mbox{Aut}(\mathbb{F}(x)/\mathbb{F})$}
\author{Richard Mandel\thanks{This work was supported by the Barnett and Jean Hollander Rich Summer Mathematics Internship, and supervised by Professor Alice Medvedev, to whom we are indebted for her invaluable guidance.}
       }
\documentclass{article}
\usepackage{amsmath}
\usepackage{amssymb}
\newtheorem{theorem}{Theorem}
\newtheorem{lemma}{Lemma}

\newtheorem{fact}{Fact}
\newtheorem{example}{Example}
\newenvironment{proof}{{\sc Proof:}}{~\hfill QED}
\newenvironment{proof1}{{\sc Proof of Theorem 1:}}{~\hfill QED}
\newenvironment{proof2}{{\sc Proof of Theorem 2:}}{~\hfill QED}

\begin{document}
\newpage
\maketitle
%%%%%%%%%%%%%%%%%%%%%%%%%%%
% abstract, keywords and Subject classification are optional.
%%%%%%%%%%%%%%%%%%%%%%%%%%%
%\begin{abstract}
%    This is a sample file.
%    You can use it as a guide for your submission.
%\end{abstract}

% Most people don't use these, so they are "commented out"
% by starting the lines with a "%"
%\begin{keywords}
%   \LaTeX, typesetting
%\end{keywords}

%\begin{AMS}
%   50C60, 18C25
%\end{AMS}

%%%%%%%%%%%%%%%%%%%%%%
% % Here is the start of the Text
%%%%%%%%%%%%%%%%%%%%%%
\section{Introduction}\label{intro}

For a finite field $\mathbb{F}$, it is a basic result of Galois theory that $\mathbb{F}(x)/\mathbb{F}$ is not a Galois extension.\footnote{While some authors consider only algebraic Galois extensions, here (in the transcendental case) we apply the more general definition used in \cite{hun80}, whereby an extension $K\subset F$ is Galois if the fixed field of $\text{Aut}(F/K)$ is $K$ itself.} In other words, if $G = \mbox{Aut}(\mathbb{F}(x)/\mathbb{F})$ and $E$ is the fixed field of $G$, then $\mathbb{F} \subsetneq E$ is a proper extension. This follows from the fact that $G$ is finite, which implies (by a well known theorem due to Artin \cite[p.~252]{hun80}) that 
\begin{eqnarray}\label{artin}
	[\mathbb{F}(x):E] = |G|,
\end{eqnarray}
whereas $[\mathbb{F}(x):\mathbb{F}] = \infty$.\footnote{This fact is the answer to an exercise in Thomas Hungerford's $\emph{Algebra}$, which motivates the present work\cite[p.~256, ex.\:9b]{hun80}. Thanks are due to Professor Khalid Bou-Rabee, in whose class the exercise was encountered.} However, constructing $E$ for a general finite field is more difficult than establishing this fact.\\

In this expository $  $paper, we construct the fixed field $E$ of $\mbox{Aut}(\mathbb{F}(x)/\mathbb{F})$ for all finite fields $\mathbb{F}$, demonstrating that it is an extension of the form $\mathbb{F}(f)$, where $f\in \mathbb{F}(x)$ is explicitly determined. The generator $f$ is shown to be easily computable in all cases using a simple formula.\\

A more concise elementary construction of $E$, due to Paul Rivoire, may be found in \cite{paul}, in which the fixed fields of the subgroups of $\mbox{Aut}(\mathbb{F}(x)/\mathbb{F})$ are also constructed. However, our exposition may be more accessible to the modern undergraduate student, and presents the generator $f$ in a simple summation formula not given in \cite{paul}. One can also construct $E$ using more sophisticated machinery from the theory of algebraic function fields; for insight into this approach, we refer the reader to \cite[p.~237, ex.\:6.9]{henning}.\\ %As the present work provides an elementary construction without recourse to such techniques, it may be of interest to the reader with limited background in this theory. For the reader interested in practical applications, the formula in Theorem \ref{formula} which generates a primitive element of $E$ over $\mathbb{F}$ may also be of use due to its simplicity.\\

The strategy will to be to look for an element $f \in E$ such that \linebreak $[\mathbb{F}(x):\mathbb{F}(f)] = |G|$, implying that $\mathbb{F}(f) = E$ by equation \eqref{artin}. 

\section{Definitions and Notation}\label{defs}
The following will hold throughout.
%Unless stated otherwise, all notation is the same as in \cite{hun80}.
\begin{enumerate}
	\item \label{defm}Let $q = p^n$, where $p$ is prime and $n \in \mathbb{N}$.
	
	\item The field $\mathbb{F}$ will always refer to $\mathbb{F}_q$, the unique (up to isomorphism) field containing $q$ elements.
	
	\item For a field $F$, $F(x)$ (resp. $F[x]$) denotes the field of rational functions (resp. ring of polynomials) in one indeterminate over $F$. The symbol `x' will only be used in this context.
	
	\item For fields $K\subset F$, $\mbox{Aut}(F/K)$ denotes the group of automorphisms of $F$ which fix $K$. Let $G = \mbox{Aut}(\mathbb{F}(x)/\mathbb{F})$.
	
	\item Let $E$ denote the fixed field of $G$.
	
	\item For $k\in \mathbb{N}$, let $f_k\in \mathbb{F}(x)$ be defined as follows: 
	
	\[f_k := \sum_{\varphi \in G} \varphi(x)^k.\]
\end{enumerate}

\begin{example}
	\em For $\mathbb{F} = \mathbb{F}_2$, %$$\Omega^k = \left \{x^k + (x+1)^k + \frac{1}{x^k} + \frac{1}{(x+1)^k} + \left (\frac{x+1}{x}\right )^k + \left (\frac{x}{x+1}\right )^k \right \}$$
	$$f_k = x^k + (x+1)^k + \frac{1}{x^k} + \frac{1}{(x+1)^k} + \left (\frac{x+1}{x}\right )^k + \left (\frac{x}{x+1}\right )^k$$
	(see Fact \ref{aut} below for a characterization of the elements of $G$).
\end{example}

\section{Main Results}

For any $\sigma \in G$, $k\in \mathbb{N}$ we have
$$\sigma(f_k) = \sigma \left ( \sum_{\varphi \in G} \varphi(x)^k \right ) = \sum_{\varphi \in G} [\sigma\! \circ\! \varphi(x)]^k = \sum_{\varphi \in G} \varphi(x)^k = f_k$$

so that $f_k \in E$. It is thus reasonable to inquire as to whether $f_k$ generates $E$ for a suitable $k$. In the following result, we will show that this is indeed the case.
\begin{theorem}\label{thm1}
	Let $m = q^2-1$. Then
	\[E = \mathbb{F}(f_m)\]
	(where $E,\mathbb{F}$ and $f_k$ are as defined in section \ref{defs})
\end{theorem}

While this theorem provides a generator for $E$ over $\mathbb{F}$, it will soon be clear that the definition of $f_m$ in section \ref{defs} is not practical for computation (at least in all but the smallest finite fields). The other main result addresses this with a surprisingly simple formula for $f_m$:
\begin{theorem}\label{formula}
	Let $$\theta_i = \begin{cases} &2 \mbox{ for }i = q,2q,\ldots ,q^2 \\ &1 \mbox{ otherwise}\end{cases}.$$
	and let
	\begin{eqnarray*}
		g &=&  \sum_{i = 0}^{q(q+1)} \theta_i x^{i(q-1)}\\
		h &=& \sum_{i = 1}^q x^{iq(q-1)}.
	\end{eqnarray*}
	Then $g,h$ are relatively prime and $f_m = g/h$.
\end{theorem}

Theorem \ref{formula} has the somewhat surprising consequence that $E\subset \mathbb{F}_p(x)$, since the coefficients of $f_m$ are shown to be either 1 or 2 (or just 1, in characteristic 2).\\

We postpone the proofs of these theorems in order to establish some preliminary facts and lemmas. Note that $m = q^2-1$ will be henceforth fixed.
\section{Facts and Lemmas}
We recall the following well known facts. In most cases, complete proofs are omitted in favor of a brief justification.
\begin{fact}\label{artinian}
	For any field extension $K \subset F$, if $H$ is a finite subgroup of $\mbox{Aut}(F/K)$ and $H'$ the fixed field of $H$, then
	$$[F:H'] = |H|.$$
\end{fact}
This is a generalization of equation \eqref{artin} in section \ref{intro}, which follows from identical reasoning.
\begin{fact}\label{div}
	For any $h\in E$, if $[\mathbb{F}(x):\mathbb{F}(h)] < \infty$, then $[\mathbb{F}(x):\mathbb{F}(h)]$ is divisible by $|G|$.
\end{fact}
This follows immediately from Fact \ref{artinian}, since  \[[\mathbb{F}(x):\mathbb{F}(h)] = [\mathbb{F}(x):E][E:\mathbb{F}(h)] = |G|[E:\mathbb{F}(h)].\]
\begin{fact}\label{frob}
	The map $\phi:\mathbb{F}(x) \rightarrow \mathbb{F}(x)$ given by $\phi(s) = s^p$ is an endomorphism of fields.
\end{fact}
It is clear that $\phi$ is an endomorphism of the multiplicative group of $\mathbb{F}(x)$, so it remains only to verify that $\phi$ is an endomorphism of the additive group. Observing that $\binom{p}{n} = 0$ for $n = 1,2,\ldots,p-1$, this follows immediately from an application of the binomial theorem to $(s+t)^p$. This is known as the Frobenius endomorphism. 
\begin{fact}\label{fermat}(Fermat's little theorem) For all $\alpha \in \mathbb{F}$, we have
	\[\alpha^q = \alpha\]
	If  $\alpha\neq 0$, we also have
	\[\alpha^{q-1} = 1\]
\end{fact}
The second identity, of which the first is an immediate consequence, follows from the fact that the multiplicative group of a finite field is cyclic \cite[p.~279]{hun80}. Since this result implies that each $\alpha \in \mathbb{F}$ is a root of $x^q-x$, we obtain the following corollary:
$$\prod_{\alpha \in \mathbb{F}} (x-\alpha) = x^q-x$$
\begin{fact}\label{sum} \[\sum_{\alpha \in \mathbb{F}} \alpha^k =  \begin{cases} -1 &\mbox{ if } (q-1)\big | k \\ \;\;\:0 &\mbox{ otherwise}\end{cases}\]
\end{fact}
\begin{proof}
	Recalling that the multiplicative group of $\mathbb{F}$ is cyclic (see Fact \ref{fermat}), let $y$ be a generator for this group and let $(q-1)\! \not \big | \:k$. Then
	\[\sum_{\alpha \in \mathbb{F}} \alpha^k =  \sum_{i=1}^{q-1} y^{ik} = \frac{y^{qk}-y^k}{y^k - 1}\]
	where the denominator on the right hand side is assured to be nonzero by our assumption that $q-1$ not divide $k$. It follows from Fact \ref{fermat} that the numerator on the right hand side is zero, proving the ``otherwise" case. If $(q-1)\big |k$, then by Fact \ref{fermat} we have
	\[\sum_{\alpha \in \mathbb{F}} \alpha^k = \sum_{\substack{\alpha \in \mathbb{F} \\ \alpha \neq 0}} 1 = q-1 = -1\]
	completing the proof.\\
\end{proof}

\begin{fact} \label{degfact}
	If $F$ is any field, $g,h\in F[x]$ are relatively prime and $g/h \not \in F$, then
	\[[F(x):F(g/h)] = \mbox{Max}\{\mbox{deg }g, \mbox{deg }h\}\]
\end{fact}

\begin{proof}
	Let $\psi \in F(g/h)[t]$ be defined as 
	\[\psi(t) = \frac{g(x)}{h(x)} h(t) - g(t)\]
	
	Then $\psi\neq 0$ (since $g/h\not \in F$) and $\psi(x) = 0$. We wish to show that $\psi$ is irreducible over $F(g/h)$.\\
	
	Since $g/h$ is transcendental over $F$ (because $g/h\not \in F$ \cite[p.~233]{hun80}) we may replace $g/h$ with the indeterminate $z$. The resulting polynomial $zh(t)-g(t)$ is primitive over $F[z]$ (because $h,g$ are relatively prime) so by Gauss' lemma \cite[p.~163]{hun80} it suffices to show that $zh(t)-g(t)$ is irreducible over $F[z]$. Let us assume it is not. As this polynomial is linear in $z$, there must be a factorization of the form 
	\[zh(t)-g(t) = [z a(t)+b(t)]c(t)\]
	
	with $a,b,c \in F[t]$, which implies that $c(t)$ divides both $g$ and $h$. But this contradicts our assumption that $g,h$ are relatively prime, so $\psi$ is irreducible over $F(g/h)$.\\
	
	Since $x$ is a root of $\psi$, which we have just shown to be irreducible over $F(g/h)$, we have
	\[[F(x):F(g/h)] = [F(g/h)(x):F(g/h)] = \mbox{deg } \psi = \mbox{Max}\{\text{deg } f,\text{deg } g\}\]
	which completes the proof.
\end{proof}

\begin{fact}\label{aut}
	For any field $F$, the elements $\varphi$ of $\mbox{Aut}(F(x)/\mathbb{F})$ are determined by equations of the form
	\[\varphi(x) = \frac{ax+b}{cx+d}\]
	where $ad \neq bc$.
\end{fact}

\begin{proof}	
	Let $\varphi \in \mbox{Aut}(F(x)/F)$, and let $g,h\in F[x]$ be relatively prime and such that $\varphi(x) = g/h$. Then $F(g/h) = \text{Im } \varphi = F(x)$, so that Fact \ref{degfact} implies that
	\[\text{Max}\{\text{deg } g,\text{deg }h\} = [F(x),F(g/h)] = 1.\]
	
	Thus, $g/h = \frac{ax+b}{cx+d}$ for some $a,b,c,d\in F$, with $ad \neq bc$ (this restriction is equivalent to the requirement that $ax+b$ and $cx+d$ be relatively prime and not both constant).\\
	
	Conversely, given $a,b,c,d\in F$ where $ad\neq bc$, let $\varphi: F(x) \rightarrow F(x)$ be the homomorphism induced by $x\mapsto \frac{ax+b}{cx+d}$. Since $\frac{ax+b}{cx+d}\not \in F$, it follows that $\frac{ax+b}{cx+d}$ is transcendental over $F$, so $\varphi$ is injective. Moreover, $\text{Im }\varphi = F(\frac{ax+b}{cx+d})$, so by Fact \ref{degfact} $[K(x):\text{Im }\varphi] = 1$, implying that $\varphi$ is surjective and hence $\varphi \in \mbox{Aut}(F(x)/F)$.
\end{proof}

\begin{fact}\label{order}
	$|G| = (q+1)q(q-1)$ 
\end{fact}
This can be demonstrated in a number of ways. For convenience, we first show that 
\begin{eqnarray}\label{PGL}G \cong \mbox{PGL}_2(\mathbb{F}),
\end{eqnarray}
from which the statement follows immediately using the well known formula for $|\mbox{PGL}_2(\mathbb{F})|$. To show that \eqref{PGL} is true, let $\varphi_{abcd}$ be the element of $G$ which maps $x \mapsto \frac{ax+b}{cx+d}$. In view of Fact \ref{aut}, it is easily verified that the map $\mu: GL_2(\mathbb{F})\rightarrow G$, where $$\left (\begin{array}{cc} a & c \\ b & d \end{array}\right ) \overset{\mu}{\longmapsto} \varphi_{abcd},$$ is a surjective homomorphism. Since $\mbox{Ker}\,\mu$ consists of all multiples of the identity, \eqref{PGL} follows immediately.\\
\subsection{}\label{subsec}
In light of Fact \ref{aut}, we may write $f_k$ as follows
\begin{eqnarray}\label{kform}
	f_k = \sum_{\substack{a,b \in \mathbb{F}\\a \neq 0}} (ax + b)^k + \sum_{\substack{a,b,c \in \mathbb{F}\\ ac \neq b}} \left(\frac{ax+b}{x+c}\right)^k.
\end{eqnarray}
Thus, letting $h_k = \prod_{c\in \mathbb{F}}(x+c)^k = (x^q-x)^k$ (where the second equation follows from the corollary to Fact \ref{fermat}), we obtain $f_k = g_k/h_k$ for an appropriate $g_k \in \mathbb{F}[x]$. If the terms of degree $k$ in the first sum on the right hand side of \eqref{kform} do not sum to zero, the highest degree term of $g$ will come from
$$h_k\sum_{\substack{a,b \in \mathbb{F}\\a \neq 0}} (ax + b)^k = (x^q-x)^k\sum_{\substack{a,b \in \mathbb{F}\\a \neq 0}} (ax + b)^k $$
so that $\mbox{deg }g_k = k(q+1)$. Thus, $[\mathbb{F}(x):\mathbb{F}(f_k)] \leq k(q+1)$ by Fact \ref{degfact}. Since we need $[\mathbb{F}(x):\mathbb{F}(f_k)] = |G| = (q+1)q(q-1)$, this means that, roughly speaking, we are looking for some $k \geq q(q-1)$ such that the high degree terms of $$\sum_{\substack{a,b \in \mathbb{F}\\a \neq 0}} (ax + b)^k$$
do not sum to zero. It turns out that most choices of $k$ do not have the desired property, because
\begin{eqnarray}\label{degofg}
	\sum_{\substack{a,b \in \mathbb{F}\\a \neq 0}} (ax + b)^k = \sum_{\substack{a\in \mathbb{F}\\a \neq 0}} a^k\sum_{b \in \mathbb{F}} (x + b/a)^k  = \left (\sum_{a\in \mathbb{F}} a^k\right )\sum_{b \in \mathbb{F}} (x + b)^k
\end{eqnarray}
(where the second equation holds because $b/a$ runs through every element of $\mathbb{F}$, and the terms in which $a=0$ are included on the right hand side as they contribute nothing to the sum) and by Fact \ref{sum} the right hand side is zero whenever $(q-1)\big|k$. Hence, $m = (q-1)(q+1)$ is a promising candidate. The following 2 lemmas formalize these notions, and will finally allow us to prove Theorem \ref{thm1}.
\begin{lemma}\label{cond}	
	If $g,h \in \mathbb{F}[x]$ are polynomials such that $g/h \in E$ and $\mbox{deg }h < \mbox{deg }g < 2|G|$, then
	$$E = \mathbb{F}(g/h).$$
\end{lemma}
\begin{proof}
	Suppose that the above conditions hold, and let $g/h = g'/h'$ with $g',h'$ relatively prime. By Fact \ref{degfact},
	\[[\mathbb{F}(x):\mathbb{F}(g/h)] = \mbox{Max}\{\mbox{deg }g', \mbox{deg }h'\} = \mbox{deg }g' \leq \mbox{deg }g < 2|G|.\]
	Since $\mbox{deg }h < \mbox{deg }g$, it follows that $g/h \in E\setminus \mathbb{F}$, so that $[\mathbb{F}(x):\mathbb{F}(g/h)] < \infty$. Thus, Fact \ref{div} implies that $$|G| \Big| [\mathbb{F}(x):\mathbb{F}(g/h)] < 2|G| $$ proving that $[\mathbb{F}(x):\mathbb{F}(g/h)] = |G| = [\mathbb{F}(x):E]$. Since $\mathbb{F}(g/h) \subset E$, the desired result follows.
\end{proof}
\begin{lemma}\label{degreelem} Let $k \in \mathbb{N}$ be a multiple of $q-1$. Then there exists a polynomial $g_k \in  \mathbb{F}[x]$ such that
	$$f_k = g_k/(x^q-x)^k$$
	and $\mbox{deg }g_k = k(q+1)$.
\end{lemma}
As this is essentially a restatement of the arguments presented at the beginning of the present section, the proof is omitted.

\section{Proof of Main Results}
\begin{proof1}
	By Lemma \ref{degreelem}, we have $$f_m = g_m/(x^q-x)^m$$ where $\mbox{deg }g_m = m(q+1) = |G| + m$ (by Fact \ref{order}). Since $|G| < |G|+m < 2|G|$, an application of Lemma \ref{cond} completes the proof.
\end{proof1}

\begin{example}
	\em Let $\mathbb{F} = \mathbb{F}_2 = \mathbb{Z}_2$, so that $|G| = 6$ and 
	\[f_m = f_3 = x^3 + (x+1)^3 + \frac{1}{x^3} + \frac{1}{(x+1)^3} + \left (\frac{x+1}{x}\right )^3 + \left (\frac{x}{x+1}\right )^3 \]
	which, with a little work, simplifies to
	\[f_m = \frac{x^6+x^5+x^3+x+1}{x^4+x^2}.\] 
\end{example}

It is clear that such computations are prohibitively complex for larger fields. In order to derive the simple formula of Theorem \ref{formula}, the following technical lemmas are required. 

\begin{lemma}\label{AB} For $A, B \in \mathbb{F}(x)$, $k\in \mathbb{N}$ we have
	\[(A-B)^{p^k-1} = A^{p^k-1}+A^{p^k-2}B + \ldots +A B^{p^k-2}+B^{{p^k-1}}\]
\end{lemma}

\begin{proof}
	If $A=B$, then
	$$A^{p^k-1}+A^{p^k-2}B+ \ldots + A B^{p^k-2}+ B^{{p^k-1}} = p^k A^{p^k-1} = 0 = (A-B)^{p^k-1}$$
	and the statement holds. If $A \neq B$, then
	\begin{eqnarray*}
		(A-B)^{p^k-1} &=& \frac{(A-B)^{p^k}}{A - B} = \frac{A^{p^k}-B^{p^k}}{A - B} = A^{p^k-1}+A^{p^k-2}B+ \ldots + A B^{p^k-2}+ B^{{p^k-1}}
	\end{eqnarray*}
	(where the second equation follows from Fact \ref{frob}) completing the proof.\\
\end{proof}
\begin{lemma}
	If $(q-1)\big |k$, then
	\begin{eqnarray}\label{tech1}
		f_k = 1 - \left (1+\sum_{b \in \mathbb{F}} (x - b)^k \right) \left ( 1 + \sum_{c \in \mathbb{F}} \frac{1}{(x - c)^k}\right ).\end{eqnarray}
\end{lemma}
\begin{proof}
	First, note that the right hand side of \eqref{tech1} is equal to 
	\begin{eqnarray}\label{plus}
		1 - \left (1+\sum_{b \in \mathbb{F}} (x + b)^k \right) \left ( 1 + \sum_{c \in \mathbb{F}} \frac{1}{(x + c)^k}\right )
	\end{eqnarray}
	since $-b,-c$ run through all elements of $\mathbb{F}$.
	Starting with equation \eqref{kform} in section \ref{subsec}, let us separate the terms in the rightmost sum according to whether or not the numerator is constant, obtaining
	$$f_k = \sum_{\substack{a,b \in \mathbb{F}\\a \neq 0}} (ax + b)^k + \sum_{\substack{b,c \in \mathbb{F}\\ b \neq 0}} \left(\frac{b}{x+c}\right)^k + \sum_{\substack{a,b,c \in \mathbb{F}\\a \neq 0 \\ ac \neq b}} \left(\frac{ax+b}{x+c}\right)^k.$$
	Applying the same reasoning as in equation \eqref{degofg} of section \ref{subsec}, the right hand side becomes
	\begin{eqnarray*}\left (\sum_{a\in \mathbb{F}} a^k \right ) \left ( \sum_{b \in \mathbb{F}} (x + b)^k + \sum_{c \in \mathbb{F}} \left(\frac{1}{x+c}\right)^k + \sum_{\substack{b,c \in \mathbb{F}\\b \neq c}} \left(\frac{x+b}{x+c}\right)^k\right )
	\end{eqnarray*}
	and by our assumption we may apply Fact \ref{sum} to obtain
	\begin{eqnarray*}
		f_k = - \sum_{b \in \mathbb{F}} (x + b)^k - \sum_{c \in \mathbb{F}} \left(\frac{1}{x+c}\right)^k - \sum_{\substack{b,c \in \mathbb{F}\\b \neq c}} \left(\frac{x+b}{x+c}\right)^k.
	\end{eqnarray*}
	which is precisely the expression obtained by expanding the right hand side of \eqref{plus}.
\end{proof}

\begin{lemma}\label{gh}
	Let us define
	\begin{eqnarray*}
		g &=& (x^q-x)^{q(q-1)}f_m \notag\\
		h &=& (x^q-x)^{q(q-1)}
	\end{eqnarray*}
	then
	\begin{enumerate}
		\item \label{poly} $g \in \mathbb{F}[x]$	
		\item \label{deg}$\mbox{deg }g = |G|$
		\item \label{relprime} $g$ and $h$ are relatively prime
		\item \label{trivial}$f_m = g/h$.
	\end{enumerate}
\end{lemma}
\begin{proof}
	Statement \eqref{trivial} is trivial, and \eqref{relprime} is a consequence of \eqref{poly} and \eqref{deg}, because if $g$ and $h$ are polynomials which are not relatively prime then (by Fact \ref{degfact}) $$[\mathbb{F}(x):\mathbb{F}(g/h)] < |G|$$ contradicting Theorem \ref{thm1}. To prove \eqref{poly} and \eqref{deg}, let us consider the expression
	\begin{eqnarray*}
		1+\sum_{b \in \mathbb{F}} (x-b)^m
	\end{eqnarray*}
	that is, the left factor in \eqref{tech1} with $k = m$. Applying Lemma \ref{AB} to expand the summands, we obtain
	\begin{eqnarray*}1 + \sum_{b \in \mathbb{F}} \sum_{i=0}^m x^{m-i}b^i  = 1 + \sum_{i=0}^m x^{m-i} \sum_{b \in \mathbb{F}} b^i.
	\end{eqnarray*}
	By Lemma \ref{sum}, the terms on the right hand side vanish except those in which  $i$ is a multiple of $q-1$. Of those terms, all become $-x^{m-i}$ except for $i=0$, which also vanishes since $\sum_{b \in \mathbb{F}} b^0 = \sum_{b \in \mathbb{F}} 1 = 0$. We thus obtain 
	\[-(x^{q(q-1)} + x^{(q-1)(q-1)} + \ldots + x^{q-1})\]
	which is equal to 
	\begin{eqnarray*}
		-(x^q-x)^{q-1}
	\end{eqnarray*}
	as may be seen by applying Lemma \ref{AB} to the latter expression. Substituting this into \eqref{tech1} (with k = m), we obtain
	\begin{eqnarray}\label{form1}f_m = 1+(x^q-x)^{q-1}\left ( 1 + \sum_{c \in \mathbb{F}} \frac{1}{(x - c)^m}\right ) = 1 + (x^q-x)^{q-1} + \sum_{c \in \mathbb{F}} \frac{(x^q-x)^{q-1}}{(x-c)^m}
	\end{eqnarray}
	and since the denominators in the sum on the right are all factors of $$\prod_{\beta \in \mathbb{F}} (x-c)^m = (x^q-x)^m = (x^q - x)^{q(q-1)}(x^q - x)^{q-1}$$
	it follows that $g = (x^q-x)^{q(q-1)}f_m \in \mathbb{F}[x]$ and $\mbox{deg }g = (q+1)q(q-1) = |G|$, proving \eqref{poly} and \eqref{deg}.
\end{proof}
\begin{lemma}\label{binom}
	\begin{eqnarray}\label{d} \sum_{i = j}^{m} \binom{i(q-1)}{j(q-1)} = \begin{cases} 0 &\mbox{for } j = 0,1,\ldots, q \\ 1 &\mbox{for } j = q\!+\!1 , q\!+\!2,\ldots, m\end{cases}
	\end{eqnarray}
	where $\binom{i(q-1)}{j(q-1)}$ denotes the binomial coefficient.
\end{lemma}
\begin{proof}
	We will make use of the following identity\footnote{This identity, and its derivation, are due to MathManiac on the math.stackexchange forum; we have generalized from $\alpha = 1$ to all $\alpha \in \mathbb{F}$ (see bibliography).}, which holds for any $\alpha \in \mathbb{F}$
	\begin{eqnarray}\label{maniac}
		\sum_{i=0}^q (x+\alpha)^{i(q-1)} = \sum_{i=0}^q x^{i(q-1)}
	\end{eqnarray}
	and may be derived as follows
	$$\sum_{i=0}^q (x+\alpha)^{i(q-1)} = \frac{(x+\alpha)^{q^2-1} - 1}{(x+\alpha)^{q-1}-1} = \frac{\frac{x^{q^2}+\alpha}{x+\alpha} - 1}{\frac{x^q+\alpha}{x+\alpha} - 1} = \frac{x^{q^2}-x}{x^q-x} = \frac{x^{q^2-1}-1}{x^{q-1}-1} = \sum_{i=0}^q x^{i(q-1)} $$
	where the second equation follows from Fact \ref{frob}.\\
	
	The left hand side of \eqref{d} is equal to the coefficient of $x^{j(q-1)}$ in
	\begin{eqnarray*}(x+1)^{m(q-1)} + (x+1)^{(m-1)(q-1)} + \ldots + (x+1)^{q-1} + 1
	\end{eqnarray*}
	which may be written as
	\begin{eqnarray*}
		\left (\sum_{i=0}^q (x+1)^{i(q-1)} \right ) \left ( \sum_{i=0}^{q-1} (x+1)^{iq(q-1)} \right ) - \sum_{i=1}^q (x+1)^{iq(q-1)}
	\end{eqnarray*}
	(where the rightmost sum corrects for a $(x+1)^{q^2(q-1)}$ term as well as ``double counted" terms). Using \eqref{maniac} on the first and third sums (as well as Fact \ref{frob} to deal with the $q$ in the powers of the third sum), this becomes
	\begin{eqnarray}\label{cake}
		\left (\sum_{i=0}^q x^{i(q-1)} \right ) \left ( \sum_{i=0}^{q-1} (x+1)^{iq(q-1)} \right ) - \sum_{i=1}^q x^{iq(q-1)}.
	\end{eqnarray} 
	Applying Fact \ref{frob} and \eqref{maniac} to the second sum yields
	\begin{eqnarray*}&&\sum_{i=0}^{q-1} (x+1)^{iq(q-1)} =  \left ( \sum_{i=0}^{q-1} (x+1)^{i(q-1)} \right )^q = \left ( \sum_{i=0}^{q} x^{i(q-1)} - (x+1)^{q(q-1)} \right )^q \\
		&=& \sum_{i=0}^{q} x^{iq(q-1)} - (x+1)^{q^2(q-1)} = \sum_{i=0}^{q} x^{iq(q-1)} - (x^{q^2}+1)^{q-1}
	\end{eqnarray*}
	whose rightmost term we may expand using Lemma \ref{AB} to obtain
	\begin{eqnarray*}
		&&\sum_{i=0}^{q} x^{iq(q-1)} - (x^{(q-1)q^2} - x^{(q-2)q^2} + \ldots + x^{2q^2} - x^{q^2} + 1) \\ &=& \sum_{i=1}^{q-1} x^{iq(q-1)} - (x^{(q-2)q^2} + \ldots + x^{2q^2} - x^{q^2}).
	\end{eqnarray*}
	Since none of the terms in parentheses have a power divisible by $q-1$, removing them from \eqref{cake} will not affect the coefficient of $x^{j(q-1)}$. Thus, the coefficient of $x^{j(q-1)}$ in \eqref{cake} is equal to the same coefficient in
	\begin{eqnarray*}
		\left (\sum_{i=0}^q x^{i(q-1)}\right ) \left ( \sum_{i=1}^{q-1} x^{iq(q-1)} \right ) - \sum_{i=1}^q x^{iq(q-1)} = x^{m(q-1)} + x^{(m-1)(q-1)}+ \ldots +  x^{(q+1)(q-1)}
	\end{eqnarray*}
	from which the result follows.
\end{proof}

We are now in a position to prove Theorem \ref{formula}.

\begin{proof2}
	
	Let $g,h$ be as in Lemma \ref{gh}. A straightforward application of Fact \ref{frob} and Lemma \ref{AB} shows that $h = \sum_{i = 1}^q x^{iq(q-1)} $, so it remains only to show that $$g =  \sum_{i = 0}^{q(q+1)} \theta_i x^{i(q-1)}.$$
	From the definition of $g$ and \eqref{form1} (in Lemma \ref{gh}) we have
	\begin{eqnarray}\label{g}
		g =  (x^q-x)^{q(q-1)} + (x^q-x)^m +  \sum_{c \in \mathbb{F}} \left (\frac{x^q-x}{x-c}\right )^m.
	\end{eqnarray}
	Observing that
	$$\frac{x^q - x}{x-c} = x^{q-1} + c x^{q-2} + \ldots + c^{q-2}x = (x^{q-1} + c x^{q-2} + \ldots + c^{q-2}x + 1) - 1$$
	an application of Lemma \ref{AB} to the right hand side yields $\frac{x^q - x}{x-c} = (x-c)^{q-1} - 1$, so that
	\begin{eqnarray*}
		\sum_{c \in \mathbb{F}} \left (\frac{x^q-x}{x-c}\right )^m = \sum_{c \in \mathbb{F}} \left ((x-c)^{q-1} - 1 \right )^m = \sum_{c \in \mathbb{F}} \sum_{i = 0}^m (x-c)^{i(q-1)} = \sum_{c \in \mathbb{F}} \sum_{i = 0}^m (x + c)^{i(q-1)}
	\end{eqnarray*}
	where the second equation follows from another application of Lemma \ref{AB}, and the third holds because $c$ runs through all elements of $\mathbb{F}$. The coefficient of $x^j$ on the right hand side is equal to
	\begin{eqnarray*}
		\sum_{c \in \mathbb{F}} \sum_{\frac{j}{q-1} \leq i \leq m} \binom{i(q-1)}{j} c^{i(q-1) - j} = \sum_{\frac{j}{q-1} \leq i \leq m} \binom{i(q-1)}{j} \sum_{c \in \mathbb{F}} c^{i(q-1) - j}  
	\end{eqnarray*}
	and by Lemma \ref{sum} the terms on the right hand side are zero whenever $(q-1) \nmid j$, and $-\binom{i(q-1)}{j}$ otherwise. Thus (with an appropriate reindexing of $j$) the coefficient of $x^{j(q-1)}$ is
	\begin{eqnarray*}\label{coeff}
		-\sum_{i = j}^m \binom{i(q-1)}{j(q-1)}
	\end{eqnarray*}
	with all other coefficients zero. Applying Lemma \ref{binom} to this expression thus produces all of the nonzero coefficients of $\sum_{c \in \mathbb{F}_q} \left (\frac{x^q-x}{x-c}\right )^m$, so that
	\[\sum_{c \in \mathbb{F}_q} \left (\frac{x^q-x}{x-c}\right )^m = -(x^{(q+1)(q-1)} + x^{(q+2)(q-1)} + \ldots + x^{m(q-1)}).\]
	Substituting this into \eqref{g}, the desired result follows from a straightforward expansion.
\end{proof2}

\begin{example}
	\em Let $\mathbb{F} = \mathbb{F}_3$. Then $|G| = 24$, $m = 8$ and by Theorem 2 we have that $E$ is generated over $\mathbb{F}$ by the element
	$$f_8 = \frac{x^{24}+x^{22}+x^{20}+2x^{18}+x^{16}+x^{14}+2x^{12}+x^{10}+x^8+2x^6+x^4+x^2+1}{x^{18}+x^{12}+x^6}.$$
\end{example}

\begin{example}
	\em Let $\mathbb{F} = \mathbb{F}_4$. Then $|G| = 60$, $m = 15$ and $E$ is generated over $\mathbb{F}$ by
	$$\frac{x^{60}+x^{57}+x^{54}+x^{51}+x^{45}+\ldots +x^{15}+x^9+x^6+x^3+1}{x^{48}+x^{36}+x^{18}+x^{12}}$$
	where the terms in the numerator with coefficient 2 (i.e. those appearing in the denominator) vanish since $\mbox{char}(\mathbb{F}_4) = 2$.
\end{example}

%\section*{About the author:}
%   Richard Mandel received his undergraduate degree in mathematics from Temple University in Philadelphia. He is currently pursuing a master's degree at City College of New York.

%\subsection*{Richard Mandel}
%   1153 3rd Ave.,
%   Apt. 15,
%   New York, NY 10065.
%   rmandel000@citymail.cuny.edu
   
\end{document}